\documentclass{siamltex}
\usepackage{amsmath,amsfonts,amssymb}
\usepackage{latexsym}
\usepackage{graphicx,overpic}
\usepackage{subfigure,caption}
\usepackage{pgfplots}
\pgfplotsset{compat=newest}
\usepackage{color}
\usepackage{booktabs}
\usepackage{caption}
\usepackage{lscape}
\usepackage{array}
\newcolumntype{C}[1]{>{\centering\let\newline\\\arraybackslash\hspace{0pt}}m{#1}}

\usepackage{tikz}

\newcommand{\enorm}{\ensuremath{| \! | \! |}}
\newcommand{\R}{\mathbb R}
\newcommand{\bA}{\mathbf A}
\newcommand{\bB}{\mathbf B}

\newcommand{\bH}{\mathbf H}
\newcommand{\bI}{\mathbf I}

\newcommand{\bP}{\mathbf P}

\newcommand{\bM}{\mathbf M}
\newcommand{\bQ}{\mathbf Q}
\newcommand{\bS}{\mathbf S}

\newcommand{\bV}{\mathbf V}

\newcommand{\bn}{\mathbf n}
\newcommand{\be}{\mathbf e}

\newcommand{\bt}{\mathbf t}
\newcommand{\bu}{\mathbf u}
\newcommand{\bU}{\mathbf U}
\newcommand{\bv}{\mathbf v}
\newcommand{\bw}{\mathbf w}

\newcommand{\bbf}{\mathbf f}
\newcommand{\A}{\mathcal A}

\newcommand{\T}{\mathcal T}

\newcommand{\tr}{{\rm tr}}
\newcommand{\Div}{\mathop{\rm div}}
\newcommand{\divG}{{\mathop{\,\rm div}}_{\Gamma}}
\newcommand{\gradG}{\nabla_{\Gamma}}

\newcommand{\gradh}{\nabla_{\Gamma_h}}
\newcommand{\gradk}{\nabla_{\Gamma_h^k}}
\newcommand{\divk}{{\mathop{\,\rm div}}_{\Gamma_h^k}}
\newcommand{\divh}{{\mathop{\,\rm div}}_{\Gamma_h}}

\newcommand{\cE}{\mathcal E}

\newcommand{\cT}{\mathcal T}

\newcommand{\nn}{\mathbb{N}}

\newcommand{\la}{\left\langle}
\newcommand{\ra}{\right\rangle}

\newtheorem{remark}{Remark}[section]

\numberwithin{equation}{section}
\begin{document}

\title{Analysis of the Taylor-Hood Surface Finite Element Method for the surface Stokes equation}
\author{
Arnold Reusken\thanks{Institut f\"ur Geometrie und Praktische  Mathematik, RWTH-Aachen
University, D-52056 Aachen, Germany (reusken@igpm.rwth-aachen.de).}
}
\maketitle

\begin{abstract} We consider the surface Stokes equation on a smooth closed hypersurface in $\R^3$. For discretization of this problem a generalization of the surface finite element method (SFEM)  of Dziuk-Elliott combined with a Hood-Taylor pair of finite element spaces has been used in the literature. We call this method Hood-Taylor-SFEM. This  method uses a penalty technique to weakly satisfy the tangentiality constraint. In this paper we present a discretization error analysis of this method resulting in optimal discretization error bounds in an energy norm. We also  address  linear algebra aspects related to (pre)conditioning of the system matrix.   
\end{abstract}
\section{Introduction}
There is a substantial recent literature on numerical approximation of the surface Stokes equations, e.g., \cite{Jankuhn1,olshanskii2018finite,OlshanskiiZhiliakov2019,Lederer2019,reusken2018stream,Brandner2020,Bonito2019a,Jankuhnetal2020,FORjointpaper,demlow2023tangential}.  In these papers different finite element techniques are treated, e.g., $H^1$-conforming methods in which the tangentiality constraint is treated by a penalty method \cite{Jankuhn1,olshanskii2018finite,OlshanskiiZhiliakov2019,Jankuhnetal2020},  $H (\Div_\Gamma)$-conforming methods combined with a Piola transformation approach \cite{Lederer2019,Bonito2019a}, discretization based on a stream function formulation \cite{reusken2018stream,Brandner2020}, or an  $H (\Div_\Gamma)$-conforming method  that avoids penalization and uses a specific construction of nodal degrees of freedom for the velocity field \cite{demlow2023tangential}.
In some of these papers rigorous discretization error analyses are presented.
There are also recent papers in which techniques used for the Stokes equations are extended  to  Navier-Stokes equations on stationary or evolving surfaces, e.g.,  \cite{nitschke2012finite,fries2018higher,reuther2018solving,Olshanskii2023}.

The conceptually maybe simplest method for discretization of  surface Stokes (or Navier-Stokes) equations is based on a natural generalization of the surface finite element method (SFEM), introduced by Dziuk-Elliott for scalar surface PDEs \cite{Dziuketal_AN_2013}, to vector-valued equations. The basic idea of this method, which has been used in the literature in e.g., \cite{fries2018higher,reuther2018solving,Reuther_Nitschke_Voigt_2020,FORjointpaper}, is as follows. Using a suitable consistent penalty term the Stokes problem on a  two-dimensional surface $\Gamma \subset \R^3$ can be written in a variational form with a velocity test and trial space, denoted by $\bV_\ast$, that contains arbitrary, i.e, \emph{not} necessarily tangential, three-dimensional velocity vectors. The tangential components of these vectors have $H^1(\Gamma)$ smoothness. The surface $\Gamma$ is approximation by a shape regular triangulation $\Gamma_h$ (for higher order approximation one can use the technique from \cite{demlow2009higher}, cf. below). On the triangular elements of $\Gamma_h$ we use a ``simple'' $H^1$-conforming pair (for velocity and pressure).  A very natural choice is the Taylor-Hood $\bP_m$-$P_{m-1}$ ($m \geq 2$) pair of finite element spaces. With this pair one can construct a Galerkin discretization of the surface Stokes variational problem in the product space $\bV_\ast \times L^2_0(\Gamma)$, with a ``variational crime'' due to the geometry approximation.  One can interpret this as a generalization of the SFEM to vector-valued problems in the sense that  one essentially discretizes the pressure and each of the three velocity components using a scalar surface finite element technique. Hence, such a method is very easy to implement if an implementation of the scalar SFEM with continuous piecewise polynomial finite elements is already available.

The main contribution of this paper is a \emph{discretization error analysis of this Taylor-Hood-SFEM}. We briefly address a few key points of the analysis. We study the general case $m \geq 2$ and thus for optimal order discretization errors we need a sufficiently accurate geometry approximation $\Gamma_h \approx \Gamma$. For this we use the parametric method introduced in \cite{demlow2009higher}. 
The polynomial order used in the parametric mapping for the geometry approximation is denoted by $k$. The case $k=1$ corresponds to a piecewise planar geometry approximation. 
 A key point in the analysis is the discrete inf-sup stability. We first show that for any $m\geq 2$ the discrete inf-sup stability property for the case $k \geq 2$ is equivalent to the discrete inf-sup stability property for the case  $k=1$. Then this property for $k=1$ is proved with  arguments that are essentially the same as in the Euclidean case, cf. \cite{Ern04,Verfuerth84}, modulo perturbations due to geometry approximation.  
Using this stability result and a Strang-Lemma, the error analysis boils down to the analysis of approximation errors for the Taylor-Hood pair and of consistency errors (caused by geometry approximation). Bounds for these errors are available in the literature. Combining these  stability, approximation and consistency results we obtain an optimal error bound in a natural energy norm. Besides this discretization error analysis we also address linear algebra aspects. We show that the penalty technique has no significant negative effect on the condition number of the system matrix. We also prove  that,  as in the standard Stokes case,   the pressure mass matrix  is an optimal preconditioner for the Schur complement matrix. 

In this paper we do not include results of numerical experiments. In the paper \cite{FORjointpaper} an extensive numerical study of the Taylor-Hood-SFEM applied to the surface Stokes equation is presented. In that paper the optimal order convergence rates of the method are demonstrated and its performance is compared with that of certain other discretization methods.  

In none of the papers mentioned above a discretization error analysis of the Taylor-Hood-SFEM is studied. In  the recent work \cite{HarderingPraetorius2023}, however,  a topic very similar to that of this paper is treated. We briefly comment on how our work is related to \cite{HarderingPraetorius2023}. The analysis in \cite{HarderingPraetorius2023} is very different from the one presented in this paper. In \cite{HarderingPraetorius2023}, for the discrete inf-sup stability analysis the macro-element technique of Stenberg \cite{Stenberg1984} is used. On the one hand this makes the analysis relatively more technical because one has  to deal with suitable equivalence classes of macro-elements. On the other hand, the analysis is more general since it applies not only to the Taylor-Hood finite elements but also to other pairs, e.g., the MINI element and the $\bP_2$-$P_0$ pair. A further difference is related to the approximation of the normal in the penalty term. In \cite{HarderingPraetorius2023} the discrete normal $\bn_h$ on  the discrete surface approximation $\Gamma_h$ is used, whereas in our setting we use an ``improved'' normal $\hat \bn_h$, cf. \eqref{betternormal} below. In \cite{HarderingPraetorius2023} this leads to a suboptimal error bound in the energy norm and  optimal error bounds  in ``tangential'' $H^1$- and $L^2$-norms for velocity and pressure, respectively. In our analysis we obtain optimal bounds in  the energy norm. In \cite{HarderingPraetorius2023} optimal $L^2$-error bounds (in a tangential norm) are derived, whereas in our paper we do not analyze $L^2$-norm error bounds. Finally we note that linear algebra aspects are not addressed in \cite{HarderingPraetorius2023}.

\section{Continuous problem} \label{sectioncont}
Let $\Gamma \subset \R^3$ be a connected compact smooth two-dimensional surface without boundary.   A tubular neighborhood of $\Gamma$ is denoted by
$
U_\delta := \left\lbrace x \in \mathbb{R}^3 \mid \vert d(x) \vert < \delta \right\rbrace,
$
with $\delta > 0$ and $d$ the signed distance function to $\Gamma$, which we take negative in the interior of $\Gamma$.  On $U_\delta$ we define $\bn(x) = \nabla d(x)$,  $\bH(x) = \nabla^2d(x)$,  $\bP = \bP(x):= \bI - \bn(x)\bn(x)^T$, and the closest point projection $\pi(x) = x - d(x)\bn(x)$. We assume $\delta>0$ to be sufficiently small such that the decomposition $
x = \pi(x) + d(x) \bn(x)
$
is unique for all $x \in U_{\delta}$. The constant normal extension of vector functions $\bv \colon \Gamma \to \mathbb{R}^3$ is defined as $\bv^e(x) := \bv(\pi(x))$, $x \in U_{\delta}$. The extension of scalar functions is defined similarly. Note that on $\Gamma$ we have $\nabla \bv^e = \nabla(\bv \circ \pi) = \nabla \bv^e \bP$, with $\nabla \bw := (\nabla w_1, \nabla w_2, \nabla w_3)^T \in \mathbb{R}^{3 \times 3}$ for smooth vector functions $\bw \colon U_\delta \to \mathbb{R}^3$. For a scalar function $g \colon U_\delta \to \mathbb{R}$ and a vector function $\bv \colon U_\delta \to \mathbb{R}^3$ we define the surface (tangential and covariant) derivatives by
\begin{equation*} \begin{split}
\nabla_{\Gamma} g(x) &= \bP(x)\nabla g(x), \quad x \in \Gamma, \\
\nabla_{\Gamma} \bv(x) &= \bP(x)\nabla \bv(x) \bP(x), \quad x \in \Gamma.
\end{split}
\end{equation*} 
If $g$, $\bv$ are defined only on $\Gamma$, we use these definitions applied to the extensions $g^e$, $\bv^e$.
On $\Gamma$  the surface strain tensor is given by
$
E(\bu):= \frac{1}{2} \left( \gradG \bu + \gradG \bu^T \right)
$.
 The surface divergence operator for vector-valued functions $\bu \colon \Gamma \to \mathbb{R}^3$ and tensor-valued functions $\bA \colon \Gamma \to \mathbb{R}^{3\times3}$ are defined as
\begin{equation*} \begin{split}
\divG \bu &:= \textrm{tr} (\gradG \bu),  \\
\divG \bA &:= \left( \divG (\be_1^T\bA), \divG (\be_2^T\bA),\divG (\be_3^T\bA) \right)^T,
\end{split}
\end{equation*}
with $\be_i$ the $i$th basis vector in $\mathbb{R}^3$. For a given force vector $\bbf \in L^2(\Gamma)^3$, with $\bbf \cdot \bn =0$, and a source term $g \in L^2(\Gamma)$, with $\int_\Gamma g\, ds =0$, we consider the following \emph{surface Stokes problem}: determine $\bu \colon \Gamma \to \R^3$ with $\bu \cdot \bn = 0$ and $p \colon \Gamma \to \R$ with $\int_\Gamma p\, ds =0$ such that
\begin{equation} \label{eqstrong} \begin{split}
- \bP \divG (E(\bu)) + \bu + \gradG p &= \bbf \qquad \text{on } \Gamma,  \\
\divG \bu &= g \qquad \text{on } \Gamma.
\end{split}
\end{equation}
We added the zero order term on the left-hand side to avoid technical details related to the kernel of the strain tensor $E$ (the so-called Killing vector fields).
The surface Sobolev space of weakly differentiable vector valued functions is denoted by
\begin{equation} \label{eqdefH1}
\begin{gathered}
\bV:= H^1(\Gamma)^3, \quad \text{with} ~ \Vert \bu \Vert_{H^1(\Gamma)}^2 := \int_{\Gamma} \Vert \bu(s) \Vert_2^2 + \Vert \nabla \bu^e(s) \Vert_2^2 \, ds.
\end{gathered}
\end{equation}
The corresponding subspace of \emph{tangential} vector field is denoted by
\begin{equation*}
\bV_T := \left\lbrace \bu \in \bV \mid \bu \cdot \bn =0 \right\rbrace.
\end{equation*}
A vector $\bu \in \bV$ can be orthogonally decomposed into a tangential and a normal part. We use the notation:
\begin{equation*}
\bu = \bP \bu + (\bu \cdot \bn)\bn =: \bu_T + u_N\bn.
\end{equation*}
For $\textbf{u}, \textbf{v} \in \bV$ and $p \in L^2(\Gamma$) we introduce the bilinear forms
\begin{align}
a(\textbf{u}, \textbf{v}) &:= \int_\Gamma E(\bu) : E(\bv) \, ds + \int_\Gamma \textbf{u} \cdot \textbf{v} \, ds, \label{blfa} \\
b(\textbf{u}, p) &:= - \int_\Gamma p \divG \bu_T \, ds. \label{blfb} 
\end{align}
Note that in the definition of $b(\bu,p)$ only the \emph{tangential} component of $\bu$ is used, i.e., $b(\bu,p)=b(\bu_T,p)$ for all $\bu \in \bV$, $p\in L^2(\Gamma)$. 
For $p \in H^1(\Gamma)$ integration by parts yields
\begin{equation}\label{Bform}
b(\bu,p)=\int_\Gamma  \bu_T\cdot \gradG p \, ds = \int_\Gamma  \bu \cdot \gradG p \, ds.
\end{equation}
We introduce the following variational formulation of \eqref{eqstrong}: determine $(\bu_T, p)  \in \bV_T \times L^2_0(\Gamma)$ such that 
\begin{equation} \label{contform} \begin{split}
a(\bu_T, \bv_T) + b(\bv_T, p) &= (\bbf, \bv_T)_{L^2(\Gamma)} ~~~ \text{for all}~ \bv_T \in \bV_T, \\
b(\bu_T,q) &= (-g,q)_{L^2(\Gamma)} ~~~ \text{for all}~ q \in L^2(\Gamma).
\end{split}
\end{equation}
The bilinear form $a(\cdot,\cdot)$ is continuous on $\bV$, hence on $\bV_T$. Ellipticity of $a(\cdot,\cdot)$ on $\bV_T$ follows from the following surface Korn inequality, that holds if $\Gamma$ is $C^2$ smooth ((4.8) in \cite{Jankuhn1}): 
 There exists a constant $c_K  \in (0,1)$ such that 
\begin{equation} \label{Korn}
\Vert \bu \Vert_{L^2(\Gamma)} + \Vert E(\bu) \Vert_{L^2(\Gamma)} \geq c_K \Vert \bu \Vert_{H^1(\Gamma)} \qquad \text{for all } \bu \in \bV_T.
\end{equation}
The bilinear form $b(\cdot,\cdot)$ is continuous on $\bV_T \times L_0^2(\Gamma)$ and satisfies the following inf-sup condition (Lemma 4.2 in \cite{Jankuhn1}):
There exists a constant $c>0$ such that estimate
\begin{equation} \label{LBBcont}
\inf_{p \in L^2_0(\Gamma)} \sup_{\bv_T \in \bV_T} \frac{b(\bv_T, p)}{\Vert \bv_T \Vert_{H^1(\Gamma)} \Vert p \Vert_{L^2(\Gamma)}} \geq c
\end{equation}
holds.
Hence, the weak formulation \eqref{contform} is \emph{a well-posed problem}.  The discretization method that we consider in this paper uses an approach in which normal velocity components are allowed but penalized in a suitable way. This method is essentially  (i.e., apart from geometric errors) a Galerkin approach applied to an \emph{extended formulation} of \eqref{contform} that we briefly discuss in the next subsection.
\subsection{Well-posed extended variational formulation} \label{secttangential}
We introduce a larger space 
$\bV_T \subset \bV \subset\bV_\ast:= \left\lbrace \bu \in L^2(\Gamma)^3 \mid \bu_T \in H^1(\Gamma)^3, u_N \in L^2(\Gamma) \right\rbrace$ and  
 bilinear forms
 \begin{align}
 k(\bu,\bv) & := \eta \int_\Gamma  (\textbf{u} \cdot \textbf{n}) ~ (\textbf{v} \cdot \textbf{n})  \, ds \qquad \bu,\bv \in \bV_*, \label{defk} \\
 A(\bu,\bv)  & :=a(\bP\bu, \bP\bv) + k(\bu, \bv) \qquad \bu,\bv \in \bV_*, \label{defA}
 \end{align}
with $\eta \geq 1$ a penalty parameter. A convenient norm on  $\bV_*$ is $\Vert  \bu \Vert_{V_*}^2 := \Vert \bu_T \Vert_{H^1(\Gamma)}^2 + \eta \Vert u_N \Vert_{L^2(\Gamma)}^2$. We then have (with $c_K$ from \eqref{Korn}):
\begin{equation} \label{Aellip}
   c_K^2 \|\bu\|_{\bV_\ast}^2 \leq A(\bu,\bu) \leq \|\bu\|_{\bV_\ast}^2 \quad \text{for all}~ \bu \in \bV_\ast.
\end{equation} 
A  penalty surface Stokes formulation is: Determine $(\bu, p)  \in \bV_* \times L^2_0(\Gamma)$ such that 
\begin{equation} \label{projectedcontform1} \begin{split}
A(\bu,\bv) + b(\bv, p) &= (\bbf, \bv)_{L^2(\Gamma)} ~~~ \text{for all}~ \bv \in \bV_*, \\
b(\bu,q) &= (-g,q)_{L^2(\Gamma)} ~~~ \text{for all}~ q \in L^2(\Gamma).
\end{split}
\end{equation} 
Note that in this formulation the vectors in the velocity space $\bV_*$ are not necessarily tangential.
The bilinear form $A(\cdot, \cdot)$ is elliptic on $\bV_\ast$, cf. \eqref{Aellip}. The inf-sup property of $b(\cdot,\cdot)$ on $\bV_\ast\times L_0^2(\Gamma)$  is an easy consequence of \eqref{LBBcont}. Using this we obtain  the following result (Theorem~6.1  in \cite{Jankuhn1}):
\begin{lemma}
Problem \eqref{projectedcontform1} is well-posed. The unique solution solves \eqref{contform}.
\end{lemma}

The variational formulation \eqref{projectedcontform1} is \emph{consistent} in the sense that its solution is the same as that of \eqref{contform}. The discretization method that we explain below is essentially a Galerkin discretization of the formulation \eqref{projectedcontform1}.

For $\bu, \bv \in \bV$, based on the  identity 
\begin{equation} \label{identity}
E(\bu) = E(\bu_T) + u_N \bH,
\end{equation}
the term $a(\bP \bu,\bP \bv)$ used in \eqref{defA} can be reformulated as 
\begin{equation} \label{defbla}
a(\bP\bu,\bP\bv) = \int_\Gamma (E(\bu) - u_N \bH) : (E(\bv) - v_N \bH) \, ds + \int_\Gamma \bP \textbf{u} \cdot \bP \textbf{v} \, ds.
\end{equation}
In this reformulation one avoids differentiation of $\bP\bu$ and $\bP\bv$ and the derivative of $\bP$ enters through $\bH$.

\section{Surface approximation and  Taylor-Hood  finite element spaces} \label{sectparametric}
For the approximation of $\Gamma$ we use the technique introduced in \cite{demlow2009higher}. We briefly explain this method and summarize results derived in that paper.

Let $\{\Gamma_h \}_{h>0}$ be a family of  polyhedrons having triangular faces  whose vertices lie on $\Gamma$ (the latter condition can be relaxed). The set of triangular faces of $\Gamma_h$ is denoted by $\T_h$ and we assume that $ \{ \T_h \}_{h>0}$ is shape regular and  quasi-uniform. The maximal diameter of the triangles $T \in \T_h$ is $h$. The outward pointing piecewise constant unit normal on $\Gamma_h$ is denoted by $\bn_h$. For $k \geq 1$ and a given $T \in \T_h$ let $\phi_1^{k}, \ldots \phi_{n_k}^{k}$ be the standard finite element Lagrange basis of polynomials of degree $k$ on $T$ corresponding to the nodal points $x_1,\ldots x_{n_{k}} \in T$. On $T$ we define 
\[
 \pi_{k}(x):= \sum_{j=1}^{n_{k}} \pi(x_j) \phi_j^{k} (x), \quad x \in T.
\]
Employing this definition on each $T \in \T_h$ yields a continuous piecewise polynomial map $\pi_{k}: \Gamma_h \to \R^3$. The image of this map is used as surface approximation
\[
  \Gamma_h^{k}:= \pi_{k}(\Gamma_h)=\{\, \pi_{k}(x)~|~x \in \Gamma_h\,\}.
\]
Note that $\Gamma_h^1=\Gamma_h$. The  outward pointing piecewise smooth   unit normal on $\Gamma_h^{k}$ is denoted by $\bn_h^{k}$ (defined a.e.) and $\bP_h^{k}:= \bI - \bn_h^k(\bn_h^k)^T$. 
The corresponding Weingarten map is $\bH_h^k:=\nabla_{\Gamma_h^k} \bn_h^k$ (defined a.e.). The accuracy of the surface approximation
$\Gamma_h^k \approx \Gamma$ increases with $k$. In \cite{demlow2009higher} the following estimates for geometric quantities are derived (for $h$ sufficiently small):
\begin{align}
 \|d\|_{L^\infty(\Gamma_h^k)} & \leq C h^{k+1}, \label{demlow1} \\
 \|\bn - \bn_h^k\|_{L^\infty(\Gamma_h^k)}& \leq C h^{k}, \label{demlow2} \\
 \|\pi- \pi_k\|_{W^{i,\infty}(T)} & \leq C h^{k+1-i}, \quad 1 \leq i \leq k, \quad T \in  \T_h,\label{demlow3} \\
 \|\bH \circ \pi -\bH_h^k\|_{L^\infty(\Gamma_h^k)} & \leq C h^{k-1}. \label{demlow4} 
\end{align}
Let the surface measures on $\Gamma$ and on $\Gamma_h^k$ be denoted by $ds$ and $ds_{hk}$, respectively, and for $x \in \Gamma_h^k$ let $\mu_{hk}(x)$ be such that $\mu_{hk}(x)ds_{hk}(x)=ds(p(x))$. In \cite{demlow2009higher} a formula for $\mu_{hk}(x)$ is derived from which the estimate
\begin{equation} \label{demlow5}
 \|1-\mu_{hk}\|_{L^\infty(\Gamma_h^k)} \leq c h^{k+1}
\end{equation}
follows. In the analysis we also need a bound for the difference between the surface measures on $\Gamma_h$ and $\Gamma_h^k$. Let $ds_h$ be the surface measure on $\Gamma_h$ and $\tilde \mu_{hk}$ such that for $x \in \Gamma_h^k$ and $\tilde x \in \Gamma_h$ with $\pi_k(\tilde x)=x$ we have $ds_{hk}(x)= \tilde \mu_{hk}(\tilde x) ds_h(\tilde x)$. Using \eqref{demlow5} and straighforward perturbation estimates we get
\begin{equation} \label{estmu}
 \|1- \tilde \mu_{hk}\|_{L^\infty(\Gamma_h)} \leq c h^{2}.
\end{equation}
For functions $v$ defined on $\Gamma_h^k$ we define an extension $v^\ell$ in  a similar way as the extension of functions defined on $\Gamma$, namely by constant extension in the  normal direction $\bn$. For scalar functions $v$ on $\Gamma_h^k$ we define (a.e.) the surface derivative by $\gradk v:= \bP_h^k \nabla v^\ell$.  For vector valued functions $\bv$  on $\Gamma_h^k$ we define  $\gradk \bv:= \bP_h^k \nabla v^\ell \bP_h^k$. If $k=1$ we write $\gradh=\nabla_{\Gamma_h^1}$.
We now relate surface derivatives on $\Gamma_h$ and $\Gamma_h^k$, $k \geq 2$. For a function $v$ on $\Gamma_h^k$ that is differentiable at $x \in \Gamma_h^k$ we have, with $\tilde x=\pi_k^{-1}(x)$ and $\tilde v (\tilde x):=v(x)$
\[
  \gradh \tilde v(\tilde x)= \bP_h(\tilde x) \nabla \pi_k^\ell (\tilde x) \bP_h^k(x) \gradk v(x).
\]
Using $\nabla \pi= \bP- d \bH$ and the estimates \eqref{demlow1}-\eqref{demlow3} one obtains
\begin{equation} \label{gradbound1}
 \big\| \gradk v(x)- \gradh v(\pi_k^{-1}(x))\big\| \leq c h \big\|\gradk v(x)\big\|, \quad x \in \Gamma_h^k,
\end{equation}
with a constant $c$ independent of $h$, $x$, $v$.  With similar arguments one obtains for a vector valued function $\bv$ on $\Gamma_h^k$
\begin{equation} \label{gradbound2}
 \big\| \gradk \bv(x)- \gradh \bv(\pi_k^{-1}(x))\big\| \leq c h \big\|\gradk \bv(x)\big\|, \quad x \in \Gamma_h^k.
\end{equation}
These results imply norm equivalences, cf. \cite{demlow2009higher}:
\begin{equation} \label{normeq} \begin{split}
  \|v \|_{H^1(\Gamma_h^k)} & \sim \|v \circ \pi_k\|_{H^1(\Gamma_h)}, ~v \in H^1(\Gamma_h^k),\\  \|\bv \|_{H^1(\Gamma_h^k)}  & \sim \|\bv \circ \pi_k\|_{H^1(\Gamma_h)} \quad \bv \in H^1(\Gamma_h^k)^3,
  \end{split}
\end{equation}
where the constants in $\sim$ can be chosen independent of $h$.

We introduce the parameterized Taylor-Hood pair on the approximate surface $\Gamma_h^k$. 
For $m \in \nn $ let $V_h^m$ be the standard Lagrange $H^1$-conforming finite element space on $\Gamma_h$, i.e.,
$V_h^m:= \{ \chi \in C(\Gamma_h)~|~ \chi_{|T} \in P_m \quad \text{for all}~ T \in \T_h\,\}$. The Taylor-Hood pair on $\Gamma_h$ is given by the velocity-pressure pair $\tilde \bV_h \times \tilde Q_h$, with $\tilde \bV_h :=(V_h^m)^3$, $\tilde Q_h:=V_h^{m-1}$, $m \geq 2$. We define the corresponding Taylor-Hood pair on $\Gamma_h^k$ by lifting these spaces to $\Gamma_h^k$ using $\pi_k$:
\begin{equation} \label{THspace} \begin{split}
  \bV_h & := \{\, \bv_h  \in C(\Gamma_h^k)^3~|~ \bv_h \circ \pi_k^{-1} = \tilde \bv_h \quad \text{for a}~~\tilde \bv_h \in \tilde \bV_h \,\},\\
  Q_h & := \{\, q_h  \in C(\Gamma_h^k)~|~ q_h \circ \pi_k^{-1} = \tilde q_h \quad \text{for a}~~\tilde q_h \in \tilde Q_h \,\}.
\end{split} \end{equation}
 Note that these spaces depend on $k$ (degree used in geometry approximation) and on $m$ (degree used in Taylor-Hood pair). 
\section{Discrete problem}  Define
\begin{align*}
 E_h(\bu) &:= \frac12 \big(\gradk \bu + \gradk \bu^T\big), \quad E_{T,h}(\bu):=E_h(\bu) - (\bu \cdot \bn_h^k) \bH_h^k, \\
 a_{h}(\bu,\bv) &:= \int_{\Gamma_h^k} E_{T,h}(\bu):E_{T,h}(\bv)\, ds_{hk} + \int_{\Gamma_h^k} \bP_h^k \bu \cdot \bP_h^k \bv \, ds_{hk},\\ 
  b_h(\bu,q)& := \int_{\Gamma_h^k} \bu \cdot \gradk q \, ds_{hk},\\
 k_h(\bu,\bv)&:= \eta \int_{\Gamma_h^k} (\bu \cdot \hat{\bn}_h^k) (\bv \cdot \hat{\bn}_h^k)  \, ds_{hk}, \\
 A_{h}(\bu,\bv) & := a_{h}(\bu,\bv) +k_h(\bu,\bv).
\end{align*}
Based on the literature, we take a penalty parameter with scaling $\eta \sim h^{-2}$. For simplicity, in the remainder we take
\begin{equation} \label{scaleta}
 \eta = h^{-2}.
\end{equation}
The reason that we introduce yet another normal approximation $\hat{\bn}_h^k$  in the penalty bilinear form $k(\cdot,\cdot)$ is the following. From the literature \cite{jankuhn2021,Hardering2022} it is known  that for obtaining optimal order error estimates (in the full energy norm)  for vector-Laplace problems, the normal  used in the penalty term has to be a more accurate approximation of the exact normal $\bn$ than $\bn_h^k$. In the remainder we assume
\begin{equation} \label{betternormal}
  \|\bn - \hat \bn_h^k\|_{L^\infty(\Gamma_h^k)} \leq C h^{k+1}. 
\end{equation}
For simplicity we assume $\hat \bn_h^k \in \bV_h$. 
\begin{remark} \rm The use of the higher order approximation $\hat{\bn}_h^k$ can be avoided in the following sense. 
In \cite{Hardering2022} it is shown (for a vector-Laplace problem) that if one uses $\bn_h^k$ instead of  $\hat{\bn}_h^k$ in the penalty term and $\eta \sim  h^{-1}$ then  optimal bounds for the \emph{tangential} error hold. For the Stokes problem this is analyzed in \cite{HarderingPraetorius2023}. 
\end{remark}

As a discrete analogon of  $ E(\bP\bu)$ we  use $E_{T,h}(\bu)=E_h(\bu) - (\bu \cdot \bn_h^k) \bH_h^k$ instead of $E_h(\bP_h^k\bu)$, cf. \eqref{identity}-\eqref{defbla}. The reason for this is that $\bP_h^k\bu$ is in  the broken space  $\cup_{T \in \cT_h} H^1(\pi_k(T))^3$ but in general not in $H^1(\Gamma_h^k)^3$ and in the analysis of the discrete problem below it is convenient to avoid the use of the broken space. This, however, is a minor technical issue.
For a suitable (sufficiently accurate) extension of the data $\bbf$ and $g$ to $\Gamma_h^k$, denoted by $\bbf_h$ and $g_h$, with $\int_{\Gamma_h^k} g_h \, ds_{hk}=0$, the finite element method reads:
Find $(\bu_h, p_h) \in \bV_h \times Q_h$, with $\int_{\Gamma_h^k} p_h \, ds_{kh}=0$, such that
\begin{equation} \label{discreteform1}
\begin{aligned}
A_h(\bu_h,\bv_h) + b_h(\bv_h,p_h) & =(\mathbf{f}_h,\bv_h)_{L^2(\Gamma_h^k)} &\quad &\text{for all } \bv_h \in \bV_h \\
b_h(\bu_h,q_h)  & = (-g_h,q_h)_{L^2(\Gamma_h^k)} &\quad &\text{for all }q_h \in Q_h.
\end{aligned}
\end{equation}
A few implementation aspects of this discretization  are briefly addressed in 
Section~\ref{secNumerical}. 

\section{Error analysis} \label{secInfsup}
In the analysis below we often write $x \lesssim y$ to state
that the inequality $x \leq  cy$ holds for quantities $x, y$ with a constant $c$ independent of $h$. Similarly for $x \gtrsim y$,
and $x \sim$ will mean that both $x \lesssim y$ and $x \gtrsim y$  hold. We introduce the norm 
\[ \enorm \bv \enorm_k^2:= \|\bv\|_{H^1(\Gamma_h^k)}^2+ h^{-2}\|\bn \cdot \bv\|_{L^2(\Gamma_h^k)}^2, \quad \bv \in H^1(\Gamma_h^k)^3.
\]
Besides the bilinear form $b_h(\bv,q)=\int_{\Gamma_h^k} \bv \cdot \gradk q \, ds_{hk}$ we also need
\[
  b_h^\ast(\bv,q):= - \int_{\Gamma_h^k} \divk \bv \, q \, ds_{hk}.
\]
 To describe the relation between these two we introduce some further notation. Denote by $\cE_h$ the collection of all edges in the curved triangulation $\pi_k(\cT_h)$ that forms $\Gamma_h^k$. For $E \in \cE_h$ the two co-normals, corresponding to the two curved elements that have $E$ as common edge, are denoted by $\nu_h^+$ and $\nu_h^{-}$ and $[\nu_h]:=\nu_h^+ + \nu_h^{-}$ (defined on $E$). Note that if the surface $\Gamma_h^k$ would be $C^1$ at $E$ then $[\nu_h] =0$. This, however, does not hold in our case and we get the following partial integration  identity
\begin{equation} \label{PI}
 b_h(\bv,q)=b_h^\ast(\bv,q)+ \sum_{T \in \cT_h} \int_{\pi_k(T)} (\bv\cdot \bn_h^k) q \, \divk \bn_h^k \, ds_{hk} + \sum_{E \in \cE_h} \int_E [\nu_h]\cdot \bv \, q \, d\ell,
\end{equation}
for functions $\bv \in H^1(\Gamma_h^k)^3$, $q\in H^1(\Gamma_h^k)$. 
In the following  lemma we collect some estimates that are useful for the error analysis.
\begin{lemma} \label{lemma2}
 For $\bv, \bw \in H^1(\Gamma_h^k)^3$, $q \in H^1(\Gamma_h^k)$ the following holds:
 \begin{align} 
  \left| a_{h}(\bv,\bw) - a(\bP\bv^\ell,\bP \bw^\ell)\right| &  \lesssim h^k \enorm \bv \enorm_k \enorm \bw \enorm_k,  \label{u1} \\
   \left| b_h(\bv,q)- b_h^\ast(\bv,q)\right| & \lesssim h \enorm \bv\enorm_k \|q\|_{L^2(\Gamma_h^k)},  \label{estbdifference} \\
   \left| b_h(\bv,q)- b(\bv^\ell,q^\ell)\right| & \lesssim h^{k}  \| \bv\|_{L^2(\Gamma_h^k)} \|q\|_{H^1(\Gamma_h^k)}, \label{estbdifference1}\\
   \left| b_h(\bv,q)- b(\bv^\ell,q^\ell)\right| & \lesssim h^{k}  \|\bv\|_{H^1(\Gamma_h^k)} \|q\|_{L^2(\Gamma_h^k)} \quad \text{if}~~\bP\bv=\bv. \label{estbdifference2}
 \end{align}
\end{lemma}
\begin{proof} 
The result \eqref{u1} is derived in  \cite[Lemma 5.16]{jankuhn2021}, \cite[Lemma 4.11]{Hardering2022}.
For the proof of \eqref{estbdifference} we use the partial integration identity \eqref{PI}.
 With \eqref{demlow2} and the definition of $\enorm \cdot \enorm_k$ we get  
 \begin{equation} \label{HH7} \begin{split}
  \left|\sum_{T \in \cT_h} \int_{\pi_k(T)} (\bv\cdot \bn_h^k) q \, \divk \bn_h^k \, ds_{hk} \right| &  \lesssim \|\bv\cdot \bn_h^k\|_{L^2(\Gamma_h^k)} \|q\|_{L^2(\Gamma_h^k)} \\ & \lesssim h \enorm \bv\enorm_k \|q\|_{L^2(\Gamma_h^k)}.
 \end{split} \end{equation}
 For the other term in the partial integration identity we note that the estimates $\|[\nu_h]\|_{L^\infty(\cE_h)}\lesssim h^k$ and $\|\bP[\nu_h]\|_{L^\infty(\cE_h)}\lesssim h^{2k}$ hold, cf. \cite[Lemma 3.5]{ORXimanum}, \cite[Lemma 7.12]{JankuhnThesis}. Using this and a standard trace estimate we obtain
 \begin{equation} \label{HH8} \begin{split}
       & \left|\sum_{E \in \cE_h} \int_E [\nu_h]\cdot \bv \, q \, d\ell \right|  \lesssim  h^k \sum_{E \in \cE_h}\int_E |\bn \cdot\bv|\, | q | \,  d\ell   + h^{2k}\sum_{E \in \cE_h} \int_E \| \bv\|\, | q | \,  d\ell \\
        & \lesssim  h^{k-1} \|\bn \cdot \bv\|_{L^2(\Gamma_h^k)} \|q\|_{L^2(\Gamma_h^k)} + h^{2k-1} \|\bv\|_{L^2(\Gamma_h^k)} \|q\|_{L^2(\Gamma_h^k)}  \\
       & \lesssim h^{k} \enorm \bv\enorm_k \|q\|_{L^2(\Gamma_h^k)}.
                              \end{split}
\end{equation}
Using the estimates  \eqref{HH7} and \eqref{HH8} in \eqref{PI} yields the result \eqref{estbdifference}. The result \eqref{estbdifference1} follows from the relation  $\gradk q = \bP_h^k(\bI - d \bH) \gradG q^\ell \circ \pi$, \eqref{demlow2}, \eqref{demlow5} and standard estimates. For the result \eqref{estbdifference2} we first note that if in \eqref{estbdifference} we restrict to $\bv$ with $\bP\bv=\bv$ then using $\|\bP \bn_h^k\|_{L^\infty(\Gamma_h^k)} \lesssim h^k$ and $\bn \cdot \bP \bv=0$, the  estimates in \eqref{HH7}-\eqref{HH8} can be improved and we obtain
\begin{equation} \label{HH9}
 \left| b_h(\bv,q)- b_h^\ast(\bv,q)\right|  \lesssim h^k \|\bv\|_{L^2(\Gamma_h^k)} \|q\|_{L^2(\Gamma_h^k)}.
\end{equation}
Using $\divk \bv = \tr(\gradk \bv)$, $\divG \bv^\ell = \tr (\gradG \bv^\ell)$ and an estimate similar to \eqref{gradbound2} we obtain
\begin{align*} 
 \left| b_h^\ast(\bv,q)-b(\bv^\ell,q^\ell)\right| &  = \left| \int_{\Gamma_h^k} \divk \bv \, q \, ds_{hk} - \int_\Gamma \divG \bv^\ell \, q^\ell \, ds\right| \\ &  \lesssim h^k \|\bv\|_{H^1(\Gamma_h^k)} \|q\|_{L^2(\Gamma_h^k)}. 
\end{align*}
Combining this with \eqref{HH9} proves the estimate \eqref{estbdifference2}.
\end{proof}

\subsection{Ellipticity property}
The following lemma shows that with  the norm $\enorm \cdot \enorm_k$  we obtain (on $\bV_h$) an analogon of the norm equivalence \eqref{Aellip}. 
\begin{lemma} \label{lemma1}
 For $h$ sufficiently small we have:
 \begin{align} 
  A_h(\bv,\bv) &\lesssim \enorm \bv \enorm_k^2, \quad \bv \in H^1(\Gamma_h^k),  \label{est1up} \\
  \enorm \bv_h  \enorm_k^2 & \lesssim  A_h(\bv_h,\bv_h), \quad \bv_h \in \bV_h. \label{est1low}
 \end{align}
\end{lemma}
\begin{proof}
For $\bv \in H^1(\Gamma_h^k)$  we have
\[
  A_h(\bv,\bv) = \|E_{T,h}(\bv)\|_{L^2(\Gamma_h^k)}^2 + \|\bP_h^k \bv\|_{L^2(\Gamma_h^k)}^2 + h^{-2}\|\hat \bn_h^k\cdot \bv\|_{L^2(\Gamma_h^k)}^2. 
 \]
With $\|\bn - \hat \bn_h^k\|_{L^\infty(\Gamma_h^k)} \lesssim h^{k+1}$, cf. \eqref{betternormal}, we get $h^{-2}\|\hat \bn_h^k\cdot \bv\|_{L^2(\Gamma_h^k)}^2 \lesssim h^{-2}\|\bn \cdot \bv\|_{L^2(\Gamma_h^k)}^2  + \|\bv\|_{L^2(\Gamma_h^k)}^2$. From this and $\|E_{T,h}(\bv)\|_{L^2(\Gamma_h^k)} \lesssim \|\bv\|_{H^1(\Gamma_h^k)}$  we get the estimate  in \eqref{est1up}. For the estimate in \eqref{est1low} we first note $ h^{-2}\|\bn \cdot \bv\|_{L^2(\Gamma_h^k)}^2 \lesssim h^{-2} \|\hat \bn_h^k \cdot \bv\|_{L^2(\Gamma_h^k)}^2 + \|\bv\|_{L^2(\Gamma_h^k)}^2$ and thus
\begin{equation} \label{HH12}
  \enorm \bv  \enorm_k^2 \lesssim \|\bv\|_{H^1(\Gamma_h^k)}^2 + A_h(\bv,\bv).
\end{equation}
For estimatating $\|\bv\|_{H^1(\Gamma_h^k)}$ we use 
  the surface Korn inequality \cite[Lemma 4.1]{Jankuhn1} and \eqref{u1}:
\begin{equation} \label{89} \begin{split}
\|\bv\|_{H^1(\Gamma_h^k)}^2 & \lesssim \|\bP \bv^\ell\|_{H^1(\Gamma)}^2 +\|\bn \cdot \bv^\ell\|_{H^1(\Gamma)}^2 \lesssim a(\bP \bv^\ell, \bP \bv^\ell) + \|\bn \cdot \bv^\ell\|_{H^1(\Gamma)}^2 \\
     & \lesssim A_h(\bv,\bv) +h^k \enorm \bv \enorm_k^2 + \|\bn \cdot \bv^\ell\|_{H^1(\Gamma)}^2.
\end{split}
\end{equation}
We insert this in \eqref{HH12} and shift (for $h$ sufficiently small) the term $h^k \enorm \bv \enorm_k^2$ to the left-hand side. We now  estimate the last term in the bound in \eqref{89}. For this we need a finite element inverse inequality (which holds also in the  parametric finite element space). Therefore we now restrict to $\bv=\bv_h \in \bV_h$.  Using a finite element inverse estimate for $\bv_h$ and for $\hat\bn_h^k \cdot \bv_h \in V_h^{2k}$ we get:
\begin{align*}
 \|\bn \cdot \bv_h^\ell\|_{H^1(\Gamma)} & \sim  \|\bn \cdot \bv_h\|_{H^1(\Gamma_h^k)} \lesssim \|\bn \cdot \bv_h\|_{L^2(\Gamma_h^k)} + \|\gradk (\bn \cdot \bv_h)\|_{L^2(\Gamma_h^k)} \\
 &  \lesssim \| \bv_h\|_{L^2(\Gamma_h^k)} + \|\gradk (\hat \bn_h^k \cdot  \bv_h)\|_{L^2(\Gamma_h^k)} \\
  & \lesssim \| \bv_h\|_{L^2(\Gamma_h^k)} + h^{-1} \|\hat \bn_h^k \cdot  \bv_h\|_{L^2(\Gamma_h^k)}\lesssim A(\bv_h,\bv_h)^\frac12.
\end{align*}
Combining these results completes the proof of \eqref{est1low}. 
\end{proof}

\subsection{Discrete inf-sup property}
For the inf-sup property we introduce $Q_{h,0}:= \{\, q_h \in Q_h~|~\int_{\Gamma_h^k} q_h \, ds_{hk}=0\,\}$. Our aim is to derive the following discrete inf-sup property:
\begin{equation} \label{infsup}
  \sup_{\bv_h \in \bV_h} \frac{b_h(\bv_h,q_h)}{\enorm \bv_h \enorm_k} \geq c_\ast \| q_h\|_{L^2(\Gamma_h^k)} \quad \text{for all}~q_h \in Q_{h,0},
\end{equation}
with $c_\ast >0$ \emph{in}dependent of $h$. Recall that the finite element spaces $\bV_h$ and $Q_h$ depend on $k$ and $m$. This inf-sup property is denoted by {\sc inf-sup}$(b_h,k,m)$, where the $b_h$ in this notation refers to the use of the bilinear form $b_h(\cdot,\cdot)$ in \eqref{infsup}. 

Below we relate (for $k \geq 2$) the discrete inf-sup property on $\Gamma_h^k$ to that  on $\Gamma_h^1=\Gamma_h$. For this it is convenient to introduce, for $v$ (or $\bv$) defined on $\Gamma_h^k$ the corresponding pull back to $\Gamma_h$ using $\tilde x:=\pi_k^{-1}(x)$, $x \in \Gamma_h^k$, $\tilde v(\tilde x):=v(x)$, $\tilde \bv (\tilde x):= \bv(x)$. We also use the notation $\enorm \cdot \enorm_1=:\enorm \cdot \enorm$. From \eqref{normeq} and $\|\bn\circ \pi_k^{-1} - \bn \|_{L^\infty(\Gamma_h^k)} \lesssim h$ we get the uniform (in $h,k$) norm equivalence
\begin{equation} \label{enormeq}
   \enorm \bv \enorm_k \sim \enorm \tilde \bv \enorm, \quad \bv \in H^1(\Gamma_h^k)^3.
\end{equation}

From \eqref{estbdifference} and a simple perturbation argument we obtain the following.
\begin{corollary} \label{corollary1} For $h$ sufficiently small:
\begin{equation} \label{eq1} \text{ {\sc inf-sup}$(b_h,k,m)$ holds iff {\sc inf-sup}$(b_h^\ast,k,m)$ holds.}
\end{equation}
 \end{corollary}
\ \\
For $k \geq 2$  we now relate {\sc inf-sup}$(b_h^\ast,k,m)$  to {\sc inf-sup}$(b_h^\ast,1,m)$:
\begin{equation}
 \label{infsup1}
  \sup_{\bv_h \in \tilde \bV_h} \frac{\int_{\Gamma_h} \divh \bv_h \, q_h \, ds_h}{\enorm \bv_h \enorm} \geq c_\ast \| q_h\|_{L^2(\Gamma_h)} \quad \text{for all}~q_h \in \tilde Q_{h,0},
\end{equation}
with $c_\ast >0$ independent of $h$. Recall that $\tilde \bV_h \times \tilde Q_h$ is the standard Taylor-Hood pair on $\Gamma_h$ (with velocity finite elements of degree $m$). 
\begin{lemma} \label{lemma3}
 For $h$ sufficiently small:
\begin{equation} \label{eq2} \text{ {\sc inf-sup}$(b_h^\ast,k,m)$ holds iff {\sc inf-sup}$(b_h^\ast,1,m)$ holds.}
\end{equation}
\end{lemma}
\begin{proof} 
 For functions $v$ on $\Gamma_h^k$ we use the correspondence $\tilde v (\tilde x)= v(x)$ introduced above.
 Take $\bv_h \in \bV_h$ with corresponding $\tilde \bv_h \in \tilde \bV_h$. Note that $\bv_h \to \tilde \bv_h$ and 
 $q_h \to \tilde q_h$ are bijections  $\bV_h \to \tilde \bV_h$ and $Q_h \to \tilde Q_h$, respectively. Using $\divk \bv_h = \tr (\gradk \bv_h)$, $\divh \tilde \bv_h = \tr (\gradh \tilde \bv_h)$ and the estimate \eqref{gradbound2} we get  
 \[
  \left| \divk \bv_h(x) - \divh \tilde \bv_h (\tilde x)\right| \lesssim h \|\gradk \bv_h (x)\|, \quad x \in \Gamma_h^k.
 \]
Using this and the estimate \eqref{estmu} for the change in surface measures on $\Gamma_h^k$ and $\Gamma_h$ we get  
\begin{equation} \label{KK2}
 \left| \int_{\Gamma_h^k} \divk \bv_h \, q_h \, ds_{hk}  - \int_{\Gamma_h} \divh \tilde \bv_h \, \tilde q_h \, ds_h \right| \lesssim h \enorm \bv_h \enorm_k \|q_h\|_{L^2(\Gamma_h^k)}.
\end{equation}
This estimate also holds if we replace $\enorm \bv_h \enorm_k \|q_h\|_{L^2(\Gamma_h^k)}$ by 
$\enorm \tilde \bv_h \enorm \|\tilde q_h\|_{L^2(\Gamma_h)}$, cf. \eqref{enormeq}. We have to deal with a minor technical issue related to the fact that $q_h \in Q_{h,0}$ does not necessarily imply $\tilde q_h \in \tilde Q_{h,0}$. Assume that {\sc inf-sup}$(b_h^\ast,1,m)$ holds. Take $q_h \in Q_{h,0}$, i.e., $\int_{\Gamma_h^k} q_h \, ds_{kh} =0$, with corresponding $\tilde q_h \in \tilde Q_h$. Define $c_q:=- \frac{1}{|\Gamma_h|} \int_{\Gamma_h} \tilde q_h \, ds_h$. Then we have, cf. \eqref{estmu}, $|c_q| \lesssim h^{k+1} \|q_h\|_{L^2(\Gamma_h^k)}$ and $\tilde q_h + c_q \in \tilde Q_{h,0}$. Hence, there exists $c_\ast >0$ independent of $h$ and $\bv_h \in  \bV_h$ such that 
\begin{align*}
  \int_{\Gamma_h^k} \divk \bv_h \, q_h \, ds_{hk} & 
   \geq \int_{\Gamma_h} \divh \tilde \bv_h (\tilde q_h +c_q)\, ds_h - c h\enorm \bv_h \enorm_k \|q_h\|_{L^2(\Gamma_h^k)}\\
   & \geq c_\ast \enorm \tilde \bv_h \enorm \| \tilde q_h +c_q\|_{L^2(\Gamma_h)} - c h\enorm \bv_h \enorm_k \|q_h\|_{L^2(\Gamma_h^k)} \\ & \geq (\tilde c - \hat c h)\enorm  \bv_h \enorm_k \| q_h\|_{L^2(\Gamma_h^k)}
\end{align*}
with suitable constants $\hat c$ and $\tilde c >0$ independent of $h$. It follows that {\sc inf-sup}$(b_h^\ast,k,m)$ holds.  Very similar arguments can be used to prove the implication in the other direction. 
\end{proof}
\ \\
From Corollary~\ref{corollary1} and Lemma~\ref{lemma3} we obtain the following result.
\begin{corollary} \label{corollary2} For $h$ sufficiently small:
\begin{equation} \label{equivess} \text{{\sc inf-sup}$(b_h,k,m)$ holds iff {\sc inf-sup}$(b_h,1,m)$ holds.}
\end{equation}
 \end{corollary}
\ \\
In the analysis above, to derive the result \eqref{equivess}  we use $b_h^*(\cdot,\cdot)$ because a direct application of perturbation arguments to $b_h(\cdot,\cdot)$, without the partial integration formula \eqref{PI}, does not yield satisfactory results. 
 We now show that the property {\sc inf-sup}$(b_h,1,m)$ indeed holds. The analysis is along the same lines as for the Taylor-Hood pair in Euclidean domains in $\R^d$, cf. \cite{Ern04,GuzmanOlshanskii}. 
\begin{theorem} For $h$ sufficiently small the property {\sc inf-sup}$(b_h,1,m)$  holds, i.e.:
\begin{equation} \label{infsupA}
  \sup_{\bv_h \in \tilde \bV_h} \frac{\int_{\Gamma_h} \bv_h \cdot \gradh q_h \, ds_h}{\enorm \bv_h \enorm} \geq c_\ast \| q_h\|_{L^2(\Gamma_h)} \quad \text{for all}~q_h \in \tilde Q_{h,0},
\end{equation}
with $c_\ast >0$ independent of $h$.
 \end{theorem}
\begin{proof}  First we consider an inf-sup estimate with the norm $\| q_h\|_{L^2(\Gamma_h)}$ replaced by a weaker one:
 \begin{equation} \label{infsupB}
  \sup_{\bv_h \in \tilde \bV_h} \frac{\int_{\Gamma_h} \bv_h \cdot \gradh q_h \, ds_h}{\enorm \bv_h \enorm} \geq c_\ast h \| \gradh q_h\|_{L^2(\Gamma_h)} \quad \text{for all}~q_h \in \tilde Q_{h},
\end{equation}
with $c_\ast >0$ independent of $h$. We show that this estimate holds using the same arguments used in the Euclidean case in e.g. \cite{GuzmanOlshanskii}. The set of edges in $\cT_h$ is denoted by $\cE_h$. For $E \in \cE_h$ the domain $\omega_E$ is the union of the two triangles that have $E$ as common edge.  We denote  the midpoint of $E$ by $x_E$. A unit tangent vector of $E$ is denoted by $\bt_E$ and $\phi_E$ denotes the continuous piecewise quadratic function on $\Gamma_h$ that is zero on $\partial \omega_E$, with $\phi_E(x_E)=1$ and extended by zero outside $\omega_E$. Take $q_h \in \tilde Q_h$. We define $\psi_E(x):= \phi_E(x) \big( \bt_E \cdot \gradh q_h(x) \big)$, $x \in \omega_E$. This function is zero on $\partial \omega_E$ and extended by zero outside $\omega_E$. Furthermore, due to the continuity of $\bt_E \cdot \gradh q_h$ across $E$ the function $\psi_E$ is continuous and piecewise polynomial of degree at most $m$. Furthermore, one easily verifies the estimate
\[
 \|\psi_E\|_{L^2(T)} + h \|\gradh \psi_E\|_{L^2(T)} \lesssim \|\gradh q_h\|_{L^2(T)}. 
\]
We define $\bv_h \in \tilde \bV_h$ by
\[
  \bv_h(x):= h^2 \sum_{E \in \cE_h} \psi_E(x) \bt_E, \quad x \in \Gamma_h.
\]
For $x \in T$ we have  $\bv_h(x) = \sum_{E \in (\cE_h \cap T)} \psi_E(x) \bt_E$ and thus $\bv_h(x) \cdot \bn_h=0$, 
where $\bn_h$ denotes the normal on $T$. 
For this specific choice  of $\bv_h$ we have
\begin{equation} \begin{split}
\|\bv_h\|_{H^1(\Gamma_h)}^2 & = \sum_{T \in \cT_h}  \|\bv_h\|_{H^1(T)}^2 = h^4 \sum_{T \in \cT_h} \big\|\sum_{E \in (\cE_h \cap T)} \psi_E \bt_E\big\|_{H^1(T)}^2  \\ & \lesssim h^2 \sum_{T \in \cT_h}\|\gradh q_h\|_{L^2(T)}^2 \sim h^2 \|\gradh q_h\|_{L^2(\Gamma_h)}^2, \end{split}
\end{equation}
and $h^{-1} \|\bn \cdot \bv_h\|_{L^2(\Gamma_h)}=h^{-1} \|(\bn- \bn_h) \cdot \bv_h\|_{L^2(\Gamma_h)} \lesssim \|\bv_h\|_{L^2(\Gamma_h)} \lesssim h^2 \|\gradh q_h\|_{L^2(\Gamma_h)}$ and thus
\begin{equation} \label{pp1}
   \enorm \bv_h \enorm \lesssim h \|\gradh q_h\|_{L^2(\Gamma_h)}
\end{equation}
holds. We also have
\begin{equation} \label{pp2} \begin{split}
 \int_{\Gamma_h} \bv_h \cdot \gradh q_h \, ds_h & = h^2 \sum_{T \in \cT_h} \sum_{E \in (\cE_h \cap T)}  (\bt_E \cdot \gradh q_h)^2 \phi_E \, ds_h   \\ & \sim h^2 \sum_{T \in \cT_h} \sum_{E \in (\cE_h \cap T)} (\bt_E \cdot \gradh q_h)^2 \, ds_h  \\ &  
 \sim h^2 \sum_{T \in \cT_h} \|\gradh q_h\|^2 \, ds_h \sim h^2 \|\gradh q_h\|_{L^2(\Gamma_h)}^2.
 \end{split}
\end{equation}
Combining the results in \eqref{pp1} and \eqref{pp2} completes the proof of \eqref{infsupB}. We now proceed using the inf-sup property \eqref{LBBcont} of the continuous problem and combine it with \eqref{infsupB} (``Verf\"urth trick'') and with pertubation arguments to control differences between quantities on $\Gamma_h$ and on $\Gamma$. Take $q_h \in \tilde Q_{h,0}$ and a constant $c_q$ such that $\int_\Gamma q_h^\ell + c_q \, ds=0$. Then $|c_q| \lesssim h^2 \|q_h^\ell\|_{L^2(\Gamma)}$ holds.  Due to \eqref{LBBcont} there exists $\bv= \bP\bv \in H^1(\Gamma)^3$ such that
\begin{equation} \label{CL} \begin{split}
  \int_{\Gamma} \bv \cdot \nabla_\Gamma q_h^\ell \, ds & = \|q_h^\ell+c_q\|_{L^2(\Gamma)}^2  \geq (1-ch^2)\|q_h\|_{L^2(\Gamma_h)}^2, \\
  \|\bv\|_{H^1(\Gamma)} & \lesssim \| q_h^\ell+c_q\|_{L^2(\Gamma)} \sim \| q_h\|_{L^2(\Gamma_h)}.
\end{split}    \end{equation}
We use a Clement type interpolation operator $I_h: H^1(\Gamma) \to V_h^1$ (i.e, continuous piecewiese linears on $\Gamma_h$), with properties $ \|I_h(v)\|_{H^1(\Gamma_h)} \lesssim \|v \|_{H^1(\Gamma)}$, $\|v^e - I_h(v)\|_{L^2(\Gamma_h)} \lesssim h \|v \|_{H^1(\Gamma)}$.  We now  choose $\bv_h:= I_h(\bv) \in \bV_h$ (componentwise application of $I_h$). For this $\bv_h$ we have $\|\bv\|_{H^1(\Gamma_h)} \lesssim  \|\bv \|_{H^1(\Gamma)} \lesssim \| q_h\|_{L^2(\Gamma_h)}$ and 
\[ 
\begin{split} 
h^{-1} \|\bn \cdot \bv_h\|_{L^2(\Gamma_h)} & = h^{-1}\|\bn \cdot (\bv^e -\bv_h)\|_{L^2(\Gamma_h)}  \lesssim  h^{-1}\| \bv^e -\bv_h\|_{L^2(\Gamma_h)} \\ & \lesssim  \|\bv \|_{H^1(\Gamma)} \lesssim  \| q_h\|_{L^2(\Gamma_h)},
\end{split} 
\]
and thus
\begin{equation} \label{pp7}
 \enorm \bv_h \enorm \lesssim \| q_h\|_{L^2(\Gamma_h)}
\end{equation}
holds. Define $\chi_{q_h}: = \sup_{\bw_h \in \tilde \bV_h} \frac{\int_{\Gamma_h} \bw_h \cdot \gradh q_h \, ds_h}{\enorm \bw_h \enorm}$. We use the splitting
\begin{equation} \label{pp8}
 \int_{\Gamma_h} \bv_h \cdot \gradh q_h \, ds_h =\int_{\Gamma_h} \bv^e \cdot \gradh q_h \, ds_h + \int_{\Gamma_h} (\bv_h-\bv^e) \cdot \gradh q_h \, ds_h.
\end{equation}
For the second term on the right-hand side we have using \eqref{infsupB} and \eqref{CL}:
\begin{equation} \label{55}
\left| \int_{\Gamma_h} (\bv_h-\bv^e) \cdot \gradh q_h \, ds_h \right| \lesssim h \|\bv \|_{H^1(\Gamma)} \|\gradh q_h\|_{L^2(\Gamma_h)} \lesssim \|q_h\|_{L^2(\Gamma_h)} \chi_{q_h}.
 \end{equation}
For the other term we use $\bP \bv^e = \bv^e$, $\gradh q_h(x)= \bP_h(\bI - d \bH) \gradG q_h^\ell (\pi(x))$, and thus
\[ \begin{split}
  \int_{\Gamma_h} \bv^e \cdot \gradh q_h \, ds_h & = \int_{\Gamma_h} \bv^e \cdot \gradG q_h^\ell(\pi(\cdot)) \, ds_h -
  \int_{\Gamma_h} \bv^e \cdot (\bP-\bP \bP_h \bP) \gradG q_h^\ell(\pi(\cdot)) \, ds_h \\
    & \quad -  \int_{\Gamma_h} \bv^e \cdot d \bP_h \bH \gradG q_h^\ell(\pi(\cdot)) \, ds_h.
\end{split} \]
With \eqref{demlow1}, \eqref{demlow5}, \eqref{CL} and $\|\bP-\bP \bP_h \bP\|_{L^\infty(\Gamma_h)} \lesssim h^2$ we get 
\begin{equation} \label{56} \begin{split}
 \int_{\Gamma_h} \bv^e \cdot \gradh q_h \, ds_h & \geq (1-ch^2)\|q_h\|_{L^2(\Gamma_h)}^2 - c h^2 \|\bv\|_{L^2(\Gamma)}\|\gradh  q_h\|_{L^2(\Gamma_h)} \\
  & \geq (1-ch^2)\|q_h\|_{L^2(\Gamma_h)}^2 -c h\|q_h\|_{L^2(\Gamma_h)} \chi_{q_h}. 
\end{split} \end{equation}
We insert the results \eqref{55} and \eqref{56} in \eqref{pp8} and divide by $\|q_h\|_{L^2(\Gamma_h)}$. Thus we obtain
\[
  \chi_{q_h} \geq c_1 (1-h^2)\|q_h\|_{L^2(\Gamma_h)} -c_2\chi_{q_h},
\]
with positive constants $c_1,c_2$ independent of $h$ and of $q_h$. From this the result \eqref{infsupA} follows.
\end{proof}

Hence, we have proved the following  result.
\begin{corollary} \label{corolinfsup} For  $h$ sufficiently small, the discrete inf-sup property \eqref{infsup} holds.
\end{corollary}

\subsection{Discretization error analysis} \label{sectAnalysis}
As usual, the discretization error analysis is based on a Strang type lemma which bounds the discretization error in terms of an approximation error and a consistency error.
We define the bilinear form
\[
\mathcal{A}_h((\bu, p),(\bv, q)) := A_h(\bu,\bv) + b_h(\bv, p) + b_h(\bu, q), \quad \bu,\bv \in H^1(\Gamma_h^k)^3,~p,q \in H^1(\Gamma_h^k).
\]
On the pair of velocity-pressure spaces $H^1(\Gamma_h^k)^3 \times H^1(\Gamma_h^k)$ we use the norm $\enorm \cdot \enorm_k^2 + \|\cdot\|_{L^2(\Gamma_h^k)}^2$. We have the continuity estimates
\begin{align}
  |A_h(\bu,\bv)| &  \lesssim \enorm \bu \enorm_k \enorm \bv \enorm_k \quad \text{for}~\bu,\bv \in H^1(\Gamma_h^k)^3,\\
  | b_h(\bu,p)|  &  \lesssim \enorm \bu \enorm_k \|p\|_{L^2(\Gamma_h^k)}\quad \text{for}~\bu \in H^1(\Gamma_h^k)^3, p \in H^1(\Gamma_h^k). \label{Cb}
\end{align}
The bound for $A_h(\cdot,\cdot)$ is obvious, cf. \eqref{est1up}. The estimate for $b_h(\cdot,\cdot)$ follows from \eqref{estbdifference} and obvious estimates for $b_h^\ast(\cdot,\cdot)$. In Lemma~\ref{lemma1} it is shown that  $A_h(\cdot,\cdot)$ is elliptic on $\bV_h$. Furthermore, the bilinear form $b_h(\cdot,\cdot)$ has the discrete inf-sup property  \eqref{infsup} on the pair of finite element spaces $\bV_h \times Q_{h,0}$, cf. Corollary~\ref{corolinfsup}. 
From standard saddle point theory it follows that (for $h$ sufficiently small) the discrete stability estimate
\begin{equation} \label{eqinfsupAh}
\sup_{(\bv_h,q_h)  \in \bV_h \times Q_{h,0}} \frac{\mathcal{A}_h((\bu_h,p_h), (\bv_h,q_h)) }{\left(\enorm \bv_h \enorm_k^2 + \Vert q_h \Vert_{L^2(\Gamma_h^k)}^2\right)^\frac{1}{2}} \gtrsim \left(\enorm \bu_h \enorm_k^2 + \Vert p_h \Vert_{L^2(\Gamma_h^k)}^2\right)^\frac{1}{2}
\end{equation}
for all $(\bu_h,p_h) \in \bV_h \times Q_{h,0}$ holds. This and the continuity of  $\mathcal{A}_h(\cdot,\cdot)$ yield the following Strang-Lemma. Here and in the remainder we use that  the solution $(\bu_T, p)  \in \bV_T \times L^2_0(\Gamma)$ of \eqref{contform} is sufficiently regular, in particular $p \in H^1(\Gamma)$.

\begin{lemma}[Strang-Lemma] \label{stranglemma}
Let $(\bu_T, p)  \in \bV_T \times L^2_0(\Gamma)$ be the  solution of problem \eqref{contform} and $(\bu_h, p_h) \in \bU_h \times Q_{h,0}$ the solution of the finite element problem \eqref{discreteform1} . The following discretization error bound holds:
\begin{multline}\label{eqStrang}
\enorm\bu_T^e  - \bu_{h} \enorm_k + \Vert p^e - p_{h} \Vert_{L^2(\Gamma_h^k)}
\lesssim \min_{(\bv_h,q_h)  \in \bV_h \times Q_{h,0}} \left( \enorm\bu_T^e - \bv_h \enorm_k + \Vert p^e - q_{h} \Vert_{L^2(\Gamma_h^k)}\right) \\
+ \sup_{(\bv_h,q_h)  \in \bV_h \times Q_{h,0}} \frac{\vert \mathcal{A}_h((\bu_T^e,p^e), (\bv_h,q_h)) - (\bbf_h, \bv_h)_{L^2(\Gamma_h^k)} + (g_h, q_h)_{L^2(\Gamma_h^k)} \vert}{\left(\enorm \bv_h \enorm_k^2 + \Vert q_h \Vert_{L^2(\Gamma_h^k)}^2\right)^\frac{1}{2}}.
\end{multline}
\end{lemma}
\ \\
Concerning the approximation error term in  the Strang-Lemma we note the following. Standard Lagrange finite element theory, cf. also \cite{demlow2009higher}, yields that for the parametric space $\pi_k(V_h^m):=\{\, v_h \in C(\Gamma_h^k)~|~v_h \circ \pi_k^{-1} \in V_h^m\,\}$, with  $V_h^m$ the standard Lagrange space on $\Gamma_h$ (polynomials of degree $m$), we have,
for $v \in H^{r+1}(\Gamma)$, 
\[
   \min_{v_h \in \pi_k(V_h^m)} \big(\|v^e - v_h\|_{L^2(\Gamma_h^k)} + h\|\gradk (v^e-v_h)\|_{L^2(\Gamma_h^k)}\big) \lesssim h^{r+1} \|v\|_{H^{r+1}(\Gamma)}, \quad  0 \leq r \leq m.
\]
Using this and the definition $\enorm \bv \enorm_k^2:= \|\bv\|_{H^1(\Gamma_h^k)}^2+ h^{-2}\|\bn \cdot \bv\|_{L^2(\Gamma_h^k)}^2$ one obtains for  $1 \leq r \leq m$, 
provided $\bu_T \in H^{r+1}(\Gamma)^3$ and $p \in H^{r}(\Gamma)$,  the following optimal approximation error bound:
\begin{equation} \label{Eq2}
\min_{(\bv_h,q_h) \in \bV_h \times Q_{h,0}} \left( \enorm \bu_T^e - \bv_h \enorm_k + \Vert p^e - q_h \Vert_{L^2(\Gamma_h^k)} \right)
\lesssim   h^{r} \left(\Vert \bu_T \Vert_{H^{r+1}(\Gamma)} + \Vert p \Vert_{H^{r}(\Gamma)}\right).
\end{equation}

We now consider  the consistency term on the right-hand side of \eqref{eqStrang}. 
We define, for $\bv, \bw \in H^1(\Gamma_h^k)^3$, $q \in H^1(\Gamma_h^k)$:
\begin{align*}
 G_a(\bv,\bw) &:= a_{h}(\bv,\bw) - a(\bP\bv^\ell,\bP \bw^\ell), \\
 G_{b}(\bv,q) &:= b_h(\bv, q) - b(\bv^\ell, q^\ell), \\ G_f(\bw) &  := (\bbf,\bw^\ell)_{L^2(\Gamma)} - (\bbf_h,\bw)_{L^2(\Gamma_h^k)}, \\
  G_g(q)   &:= (g_h,q)_{L^2(\Gamma_h^k)} - (g,q^\ell)_{L^2(\Gamma)}.
\end{align*}
Let $(\bu_T, p)  \in \bV_T \times L^2_0(\Gamma)$ be the unique solution of problem \eqref{contform} and $(\bv_h,q_h) \in \bU_h \times Q_h$.
The consistency term in \eqref{eqStrang} can be written as
\begin{equation} \label{consterm1} \begin{split}
&\mathcal{A}_h((\bu_T^e,p^e), (\bv_h,q_h)) - (\bbf_h, \bv_h)_{L^2(\Gamma_h^k)} + (g_h, q_h)_{L^2(\Gamma_h^k)}  \\
& = A_h(\bu_T^e, \bv_h) + b_h(\bv_h,p^e) + b_h(\bu_T^e, q_h)  - (\bbf_h, \bv_h)_{L^2(\Gamma_h^k)} + (g_h, q_h)_{L^2(\Gamma_h^k)}  \\
&\qquad + \underbrace{(\bbf, \bP \bv_h^\ell)_{L^2(\Gamma)} -  (g, q_h^\ell)_{L^2(\Gamma)}  -  a(\bu_T,\bP\bv_h^\ell) - b(\bP\bv_h^\ell, p) - b(\bu_T, q_h^\ell)}_{=0} \\
&= G_a(\bu_T^e, \bv_h) + G_b(\bv_h,p^e) + G_b(\bu_T^e, q_h)+  k_h(\bu_T^e, \bv_h) + G_f(\bv_h) + G_g(q_h).
\end{split}
\end{equation}
From \eqref{u1} and $\enorm \bu_T^e \enorm_k = \|\bu_T^e\|_{H^1(\Gamma_h^k)} \sim \|\bu_T \|_{H^1(\Gamma)}$ we get
\begin{equation} \label{E1}
 | G_a(\bu_T^e, \bv_h)| \lesssim h^k \|\bu_T\|_{H^1(\Gamma)} \enorm \bv_h\enorm_k.
 \end{equation}
Using \eqref{estbdifference1} and \eqref{estbdifference2} we obtain 
\begin{equation} \label{E2}
 |G_b(\bv_h,p^e)| \lesssim h^k \enorm \bv_h\enorm_k \|p\|_{H^1(\Gamma)}, \quad  |G_b(\bu_T^e, q_h)|  \lesssim h^k \|\bu_T\|_{H^1(\Gamma)} \|q_h\|_{L^2(\Gamma_h^k)}.
\end{equation}
For the penalty term we get, using \eqref{betternormal}:
\begin{equation} \label{E3} \begin{split}
 |  k_h(\bu_T^e, \bv_h) | & = h^{-2}\big| \int_{\Gamma_h^k} (\hat \bn_h^k  - \bn)\cdot\bu_T^e (\hat \bn_h^k \cdot \bv_h) \, ds_{hk}\big|
 \\
 & \lesssim h^{-1} \|\hat \bn_h^k  - \bn\|_{L^\infty(\Gamma_h^k)} \|\bu_T\|_{L^2(\Gamma)} \enorm \bv_h \enorm_k \lesssim h^k\|\bu_T\|_{L^2(\Gamma)} \enorm \bv_h \enorm_k. 
\end{split}
\end{equation}
Note that in the last estimate in \eqref{E3}, in order to obtain a bound of order $h^k$ we need the ``improved'' normal approximation $\hat \bn_h^k$ with error bound of order $h^{k+1}$. For the data errors we assume
\begin{equation} \label{dataerror}
 \| \bbf^e -  \bbf_h\|_{L^2(\Gamma_h^k)} \lesssim h^k \|\bbf\|_{L^2(\Gamma)}, \quad \|g^e - g_h |_{L^2(\Gamma_h^k)}\lesssim h^k \|g \|_{L^2(\Gamma)},
\end{equation}
which then yield the bound
\begin{equation} \label{E4}
 |G_f(\bv_h) + G_g(q_h)| \lesssim h^k \big(\|\bbf\|_{L^2(\Gamma)}\enorm \bv_h \enorm_k +\|g \|_{L^2(\Gamma)} \|q_h\|_{L^2(\Gamma_h^k)}\big). 
\end{equation}
Combining the results above we obtain the following result for the consistency term in \eqref{eqStrang}.
\begin{lemma}\label{lemconssis}
Let $\bbf_h$ and $g_h$ be approximations of $\bbf$ and $g$ such that \eqref{dataerror} holds. For the  solution $(\bu_T, p)  \in \bV_T \times L^2_0(\Gamma)$ of problem \eqref{contform} the following holds:
\begin{equation*} \begin{split}
 & \sup_{(\bv_h,q_h)  \in \bV_h \times Q_{h,0}} \frac{\vert \mathcal{A}_h((\bu_T^e,p^e), (\bv_h,q_h)) - (\bbf_h, \bv_h)_{L^2(\Gamma_h^k)} + (g_h, q_h)_{L^2(\Gamma_h^k)} \vert}{\left(\enorm \bv_h \enorm_k^2 + \Vert q_h \Vert_{L^2(\Gamma_h^k)}^2\right)^\frac{1}{2}}\\
 & \qquad \qquad \qquad \lesssim h^k\big( \|\bu_T\|_{H^1(\Gamma)} +\|p\|_{H^1(\Gamma)}+ \|\bbf\|_{L^2(\Gamma)}+ \|g \|_{L^2(\Gamma)}\big).
\end{split}
\end{equation*}
\end{lemma}
The results in Lemma~\ref{stranglemma}, \eqref{Eq2} and Lemma~\ref{lemconssis} yield the following (optimal) discretization error bound.
\begin{theorem} \label{mainthm}
 Let $(\bu_T, p)  \in H^{r+1}(\Gamma)^3 \times H^r(\Gamma)$, with $r \geq 1$, be the  solution of problem \eqref{contform} Let $(\bu_h, p_h) \in \bU_h \times Q_{h,0}$ the solution of the finite element problem \eqref{discreteform1} with data such that \eqref{dataerror} is satisfied. The following discretization error bound holds for $1 \leq r \leq m$:
\[ \label{mainest}  \begin{split}
  \enorm\bu_T^e  - \bu_{h} \enorm_k + \Vert p^e - p_{h} \Vert_{L^2(\Gamma_h^k)} &
   \lesssim  h^{r} \left(\Vert \bu_T \Vert_{H^{r+1}(\Gamma)}  + \Vert p \Vert_{H^{r}(\Gamma)}\right) \\ &  + h^k\big( \|\bu_T\|_{H^1(\Gamma)} +\|p\|_{H^1(\Gamma)}+ \|\bbf\|_{L^2(\Gamma)}+ \|g \|_{L^2(\Gamma)}\big).
\end{split}
\]
 \end{theorem}
\ \\
We expect that using the techniques as in \cite{olshanskii2018finite,Hardering2022} and the energy error bound in Theorem~\ref{mainthm} above one can derive an optimal  $L^2$-error bound. We do not study this further here. 

\section{Linear algebra aspects} \label{secNumerical}
We briefly discuss a few implementation aspects and study conditioning of the resulting stiffness matrix. In particular we show that the penalty technique that is used in the discretization does not lead to poor conditioning properties of the stiffness matrix. A nice property of the method treated in this paper is that its implementation is very straightforward if a code for higher order surface parametric finite elements (as in \cite{demlow2009higher}) for scalar problems is already available.  One can then essentially use this code for each of the three velocity components and for the pressure unknown. Using the parametrization $\pi_k: \Gamma_h \to \Gamma_h^k$, the integrals over $\Gamma_h^k$ used in the bilinear forms are reformulated as integrals over $\Gamma_h$ and the discrete velocity $\bu_h= \tilde \bu_h \circ \pi_k^{-1}$,
 $\tilde \bu_h \in (V_h^m)^3$, and discrete pressure $p_h = \tilde p_h \circ \pi_k^{-1}$, $\tilde p_h \in \tilde V_h^{m-1}$,  are determined using the standard nodal basis in $V_h^m$ and $V_h^{m-1}$, respectively. An extensive numerical study of this surface Taylor-Hood finite element method for the Stokes problem is presented in \cite{FORjointpaper}. In that paper the isoparametric case $k=m$ with $k=2,3$, i.e. the Taylor-Hood pairs $\bP_2$-$P_1$ and $\bP_3$-$P_2$, is treated. Numerical experiments presented in \cite{FORjointpaper} demonstrate optimal order convergence (both in energy and $L^2$ norms). We refer to that paper for these results and for further details on the implementation.

Note that there is some overhead in computational work due to the fact that we use a \emph{three}-dimensional discrete velocity $\bu_h$ as approximation for the two-dimensional tangential velocity $\bu= \bu_T$. The polynomials used in the finite element method, however, are all defined on two-dimensional triangular domains. For such a polynomial of degree $m$ the number of degrees of freedom is $\tfrac12 (m+1)(m+2)$. Hence, if one uses Taylor-Hood $\bP_{m}-P_{m-1}$, $m \geq 2$, for Stokes in a planar domain (i.e., two velocity components) one has per triangle in total (i.e. velocity and pressure) $(m+1)(\tfrac32 m+2)$ unknowns. In our situation here, where we use \emph{three} velocity components the total number of unknowns per triangle is $(m+1)(2 m+3)$. We thus have an overhead factor (w.r.t. number of unknowns) of $(2 m+3)/(\tfrac32 m+2) \in (1\tfrac13, 1\tfrac25]$.

For an analysis of linear algebra aspects we need some further notation.  Let $n_u>0, n_p>0$ be the number of degrees of freedom in the finite element spaces $\bV_h$ and $Q_h$, i.e., $n_u= {\rm dim}(\bV_h)$, $n_p={\rm dim}(Q_h)$. Furthermore, $P_h^V:\,\mathbb{R}^{n_u}\to \bV_h$ and $P_h^Q:\,\mathbb{R}^{n_p}\to Q_h$ are canonical mappings between the vectors of nodal values and finite element functions, using the ($\pi_k$ image of the) nodal bases in $V_h^m$ (for velocity) and in $V_h^{m-1}$ 
(for pressure).
Denote by $\la\cdot,\cdot\ra$ and $\|\cdot\|$ the Euclidean scalar product and the corresponding norm. For matrices, $\|\cdot\|$ denotes the spectral norm in this section. Now we introduce several matrices. Let
$\bA\in\mathbb{R}^{n_u\times n_u}$, $\bB\in\mathbb{R}^{n_p\times n_u}$, $\bM_u\in \mathbb{R}^{n_u \times n_u}$, $\bM_p\in \mathbb{R}^{n_p \times n_p}$ be such that
\[
\begin{split}
\la \bA \vec u, \vec v\ra &=  A_h(P_h^V \vec u, P_h^V \vec v),~ \la \bB \vec u,\vec \lambda \ra= b_h(P_h^V \vec u,P_h^Q \vec\lambda),
\\ \la \bM_u\vec u,\vec v \ra & = (P_h^V \vec u,P_h^V \vec v)_{L^2(\Gamma_h^k)}, \quad \la \bM_p\vec \lambda,\vec \mu \ra  = (P_h^Q \vec \lambda,P_h^Q \vec \mu)_{L^2(\Gamma_h^k)},
\end{split}
\]
for all $\vec u,\vec v\in\mathbb{R}^{n_u},~~\vec \mu,\,\vec \lambda\in\mathbb{R}^{n_p}$. The matrices $\bA, \bM_u$ and $\bM_p$ are symmetric positive definite. With the same arguments as in a Euclidean domain in $\R^2$ one can verify that the  mass matrices $\bM_u$ and 
$\bM_p$ have a spectral condition number that is uniformly bounded,  independent of  $h$. 
We introduce the system matrix and its Schur complement:
\[
\A:=\left[\begin{matrix}
             \bA & \bB^T  \\
             \bB & 0
           \end{matrix}\right],\quad \bS:=\bB \bA^{-1} \bB^T.
\]
Let $\mathbf{1} \in \R^{n_p}$ be the vector with all entries 1. We have $\bB^T \mathbf 1=0$, hence $\A$ is singular. 
The algebraic system resulting from the finite element method \eqref{discreteform1} has the following form: Determine $\vec u \in \R^{n_u}$, $\vec\lambda \in \R^{n_p}$ with $\langle \bM_p \vec\lambda , \mathbf{1} \rangle =0$ such that 
\begin{equation}\label{SLAE}
\A \begin{pmatrix} \vec u \\ \vec \lambda \end{pmatrix}=\vec b,\quad\text{with suitable}~ \vec b\in\mathbb{R}^{n_u+n_p}.
\end{equation}
We will consider a block-diagonal preconditioner of the matrix $\A$, as is standard for discretized Stokes problems in Euclidean domains,  e.g., \cite{Benzi,ElmanBook}. For this we first analyze spectral properties of the matrices $\bA$ and $\bS$. In the following lemma we use spectral inequalities for symmetric matrices. We use $\mathbf{1}^{\perp_M}:= \{\, \vec \lambda \in \R^{n_p}~|~\langle \bM_p \vec\lambda , \mathbf{1} \rangle =0\,\}$. 
\begin{lemma} \label{lemprecond} There are strictly positive constants $\nu_{A,1}$, $\nu_{A,2}$, $\nu_{S,1}$, $\nu_{S,2}$, independent of $h$,  such that the following spectral inequalities hold:
 \begin{align}
  \nu_{A,1}  \bM_u & \leq \bA \leq \nu_{A,2} h^{-2} \bM_u, \label{spec1} \\
  \nu_{S,1}  \bM_p & \leq \bS \leq \nu_{S,2}  \bM_p \quad \text{on}~\mathbf{1}^{\perp_M}. \label{spec2}
 \end{align}
\end{lemma}
\begin{proof}
Note that for $\vec v\in\mathbb{R}^{n_u}$ with $\bv_h:=P_h^V \vec v$ we have 
\begin{equation}\label{cond1}
\frac{\la \bA\vec v,\vec v\ra}{\la \bM_u \vec v, \vec v\ra}=
\frac{A_h(P_h^V \vec v, P_h^V \vec v)}{\|P_h^V \vec v\|^2_{L^2(\Gamma_h^k)}}=\frac{A_h(\bv_h,\bv_h)}{\|\bv_h\|^2_{L^2(\Gamma_h^k)}} .
  \end{equation}
From Lemma~\ref{lemma1} we get $A_h(\bv_h,\bv_h) \sim \enorm \bv_h \enorm_k^2 = \|\bv_h\|_{H^1(\Gamma_h^k)}^2+h^{-2}\|\bn \cdot \bv_h\|_{L^2(\Gamma_h^k)}^2$ and using a finite element inverse inequality we obtain  
\[
 \|\bv_h\|_{L^2(\Gamma_h^k)}^2 \lesssim \enorm \bv_h\enorm_k^2 \lesssim h^{-2} \|\bv_h\|_{L^2(\Gamma_h^k)}^2,
\]
which proves the estimates in \eqref{spec1}.
For the Schur complement matrix $\bS$, we have
\begin{equation}\label{cond3}
 \la \bS \vec \lambda,\vec \lambda \ra = \Big( \sup_{\bv_h \in \bV_h} \frac{b_h(\bv_h,q_h)}{A_h(\bv_h,\bv_h)^\frac12} \Big)^2, \quad q_h:=P_h^Q \vec \lambda.
  \end{equation}
Using Lemma~\ref{lemma1} and the discrete inf-sup property, cf. Corollary~\ref{corolinfsup},  we  get for $\vec \lambda \in \mathbf{1}^{\perp_M}$:
\begin{equation} 
 \la \bM_p \vec \lambda, \vec \lambda \ra = \|q_h\|_{L^2(\Gamma_h^k)}^2 \lesssim \Big( \sup_{\bv_h \in \bV_h} \frac{b_h(\bv_h,q_h)}{A_h(\bv_h,\bv_h)^\frac12} \Big)^2 = \la \bS \vec \lambda,\vec \lambda \ra ,
\end{equation}
which proves the first inequality in \eqref{spec2}. The other inequality in  \eqref{spec2} follows from \eqref{cond3} and \eqref{Cb}:
\[ \begin{split}
  |b_h(\bv_h,q_h)| & \lesssim \enorm\bv_h\enorm_k \|q_h\|_{L^2(\Gamma_h^k)} \lesssim A_h(\bv_h,\bv_h)^\frac12 \|q_h\|_{L^2(\Gamma_h^k)}\\ &  = A_h(\bv_h,\bv_h)^\frac12 \la \bM_p \vec \lambda, \vec \lambda \ra^\frac12.
\end{split} \]
\end{proof}
\ \\[1ex]
The result in \eqref{spec1} shows that the condition number of the $\bA$ matrix behaves as in a standard Stokes problem. In particular the penalty technique has no significant negative effect on the condition number of $\bA$. The result in \eqref{spec1}  shows that,  as in the standard Stokes case,   the pressure mass matrix $\bM_p$ is an optimal preconditioner for the Schur complement matrix $\bS$. 

For the analysis of block precondioners for $\A$ we can apply analyses known from the literature  \cite[Section 4.2]{ElmanBook}. These results show that for an efficient solver for the linear system \eqref{SLAE} one needs only an efficient solver for the symmetric positive definite $\bA$ block. One particular result \cite[Theorem 4.7]{ElmanBook} is the following.  
\begin{corollary}\label{corr:pc}
Define a block diagonal preconditioner
\[
  Q:=\left[\begin{matrix}
             \bQ_A & 0  \\
             0 & \bM_p
           \end{matrix}\right]
\]
of $\A$,  with $\bQ_A \sim \bA$  a uniformly spectrally equivalent preconditioner of $\bA$. For the effective spectrum $\sigma_\ast(Q^{-1}\A):= \sigma(Q^{-1}\A) \setminus \{0\}$
of the preconditioned matrix  we have
\[ \sigma_\ast(Q^{-1}\A) \subset \big([C_{-},c_{-}]\cup[c_+,C_+]\big), \]
with some constants
$C_{-} < c_{-} < 0 < c_+ < C_+$ independent of $h$.
\end{corollary}

\subsection*{Acknowledgment} The author wishes to thank Hanne Hardering and Simon Praetorius for fruitful discussions and the German Research Foundation (DFG) for financial support within the Research Unit ``Vector- and tensor valued surface PDEs'' (FOR 3013) with project no. RE 1461/11-2.

\bibliographystyle{siam}
\bibliography{literatur}{}

\begin{thebibliography}{10}

\bibitem{Benzi}
{\sc M.~Benzi, J.~Liesen, and G.~Golub}, {\em Numerical solution of saddle
  point problems}, Acta Numerica, 14 (2005), pp.~1--137.

\bibitem{Bonito2019a}
{\sc A.~Bonito, A.~Demlow, and M.~Licht}, {\em A divergence-conforming finite
  element method for the surface {Stokes} equation}, SIAM J. Numer. Anal., 58
  (2020), pp.~2764--2798.

\bibitem{FORjointpaper}
{\sc P.~Brandner, T.~Jankuhn, S.~Praetorius, A.~Reusken, and A.~Voigt}, {\em
  Finite element discretization methods for velocity-pressure and stream
  function formulations of surface {Stokes} equations}, SIAM J. Sci. Comp., 44
  (2022), pp.~1807--1832.

\bibitem{Brandner2020}
{\sc P.~Brandner and A.~Reusken}, {\em Finite element error analysis of surface
  {Stokes} equations in stream function formulation}, ESAIM: M2AN, 54 (2020),
  pp.~2069--2097.

\bibitem{demlow2009higher}
{\sc A.~Demlow}, {\em Higher-order finite element methods and pointwise error
  estimates for elliptic problems on surfaces}, SIAM J. Numer. Anal., 47
  (2009), pp.~805--827.

\bibitem{demlow2023tangential}
{\sc A.~Demlow and M.~Neilan}, {\em A tangential and penalty-free finite
  element method for the surface {Stokes} problem}, arXiv:2307.01435,  (2023).

\bibitem{Dziuketal_AN_2013}
{\sc G.~Dziuk and C.~M. Elliott}, {\em {Finite element methods for surface
  PDEs}}, Acta Numerica, 22 (2013), pp.~289--396.

\bibitem{ElmanBook}
{\sc H.~Elman, D.~Silvester, and A.~Wathen}, {\em Finite Elements and Fast
  Iterative Solvers}, Oxford University Press, Oxford, first~ed., 2005.

\bibitem{Ern04}
{\sc A.~Ern and J.-L. Guermond}, {\em Theory and Practice of Finite Elements},
  Springer, New York, 2004.

\bibitem{fries2018higher}
{\sc T.-P. Fries}, {\em Higher-order surface {FEM} for incompressible
  {Navier-Stokes} flows on manifolds}, Inter. J. Numer. Methods Fluids, 88
  (2018), pp.~55--78.

\bibitem{GuzmanOlshanskii}
{\sc J.~Guzman and M.~Olshanskii}, {\em Inf-sup stability of geometrically
  unfitted {Stokes} finite elements}, Math. Comp., 87 (2018), pp.~2091--2112.

\bibitem{Hardering2022}
{\sc H.~Hardering and S.~Praetorius}, {\em Tangential errors of tensor surface
  finite elements}, IMA J. Numer. Anal., 43 (2022), pp.~1543--1585.

\bibitem{HarderingPraetorius2023}
\leavevmode\vrule height 2pt depth -1.6pt width 23pt, {\em A parametric
  finite-element discretization of the surface {Stokes} equations},
  arXiv:2309.00931,  (2023).

\bibitem{JankuhnThesis}
{\sc T.~Jankuhn}, {\em Finite Element Methods for Surface Vector Partial
  Differential Equations}, {PhD} thesis, RWTH Aachen University, 2021.

\bibitem{Jankuhn1}
{\sc T.~Jankuhn, M.~A. Olshanskii, and A.~Reusken}, {\em Incompressible fluid
  problems on embedded surfaces: Modeling and variational formulations},
  Interfaces and Free Boundaries, 20 (2018), pp.~353--377.

\bibitem{Jankuhnetal2020}
{\sc T.~Jankuhn, M.~A. Olshanskii, A.~Reusken, and A.~Zhiliakov}, {\em Error
  analysis of higher order trace finite element methods for the surface
  {Stokes} equations}, J. Numer. Math., 29 (2021), pp.~245--267.

\bibitem{jankuhn2021}
{\sc T.~Jankuhn and A.~Reusken}, {\em Trace finite element methods for surface
  vector-{Laplace} equations}, IMA J. Numer. Analysis, 41 (2021), pp.~48--83.

\bibitem{Lederer2019}
{\sc P.~L. Lederer, C.~Lehrenfeld, and J.~Sch\"oberl}, {\em Divergence-free
  tangential finite element methods for incompressible flows on surfaces}, Int.
  J. Numer. Meth. Eng., 121 (2020), pp.~2503--2533.

\bibitem{nitschke2012finite}
{\sc I.~Nitschke, A.~Voigt, and J.~Wensch}, {\em A finite element approach to
  incompressible two-phase flow on manifolds}, Journal of Fluid Mechanics, 708
  (2012), pp.~418--438.

\bibitem{ORXimanum}
{\sc M.~Olshanskii, A.~Reusken, and X.Xu}, {\em A stabilized finite element
  method for advection-diffusion equations on surfaces}, IMA J Numer. Anal., 34
  (2014), pp.~732--758.

\bibitem{olshanskii2018finite}
{\sc M.~A. Olshanskii, A.~Quaini, A.~Reusken, and V.~Yushutin}, {\em A finite
  element method for the surface {Stokes} problem}, SIAM Journal on Scientific
  Computing, 40 (2018), pp.~A2492--A2518.

\bibitem{OlshanskiiZhiliakov2019}
{\sc M.~A. Olshanskii, A.~Reusken, and A.~Zhiliakov}, {\em Inf-sup stability of
  the trace {$P_2$-$P_1$ Taylor-Hood elements for surface PDEs}}, Math. Comp.,
  90 (2021), pp.~1527--1555.

\bibitem{Olshanskii2023}
{\sc M.~A. Olshanskii, A.~Reusken, and A.~Zhiliakov}, {\em Tangential
  {Navier-Stokes} equations on evolving surfaces: Analysis and simulations},
  Math. Mod. Methods in Appl. Sc., 14 (2022), pp.~2817--2852.

\bibitem{reusken2018stream}
{\sc A.~Reusken}, {\em Stream function formulation of surface {Stokes}
  equations}, IMA J. Numer. Anal., 40 (2020), pp.~109--139.

\bibitem{Reuther_Nitschke_Voigt_2020}
{\sc S.~Reuther, I.~Nitschke, and A.~Voigt}, {\em A numerical approach for
  fluid deformable surfaces}, Journal of Fluid Mechanics, 900 (2020), p.~R8.

\bibitem{reuther2018solving}
{\sc S.~Reuther and A.~Voigt}, {\em Solving the incompressible surface
  {Navier-Stokes} equation by surface finite elements}, Physics of Fluids, 30
  (2018), p.~012107.

\bibitem{Stenberg1984}
{\sc R.~Stenberg}, {\em Analysis of mixed finite element methods for the
  {Stokes problem: A} unified approach}, Math. Comp., 42 (1984), pp.~9--23.

\bibitem{Verfuerth84}
{\sc R.~Verf\"urth}, {\em Error estimates for a mixed finite element
  approximation of the {Stokes} equation}, {RAIRO} Anal. Numer., 18 (1984),
  pp.~175--182.

\end{thebibliography}

\end{document}